\documentclass[12pt,intlimits]{gen-j-l}
\usepackage{amssymb,amsmath,amsthm,upref,cite}

\numberwithin{equation}{section}

\theoremstyle{plain}
\newtheorem{theorem}{Theorem}[section]
\newtheorem{corollary}{Corollary}[section]
\newtheorem{lemma}{Lemma}[section]
\newtheorem{proposition}{Proposition}[section]

\theoremstyle{remark}
\newtheorem{remark}[theorem]{Remark}

\begin{document}

\title{Malmheden's Theorem Revisited}
\author{M. Agranovsky}
\address{Department of Mathematics, Bar--Ilan University, Ramat--Gan, 52900 ISRAEL}
\email{agranovs@math.biu.ac.il}

\author{D. Khavinson}
\address{Department of Mathematics \& Statistics, University of South Florida, Tampa, FL 33620-5700, USA}
\email{dkhavins@cas.usf.edu}

\author{H. S. Shapiro}
\address{Department of Mathematics, Royal Institute of Technology, Stockholm, SWEDEN 100 44}
\email{shapiro@telia.com,shapiro@math.kth.se}
\thanks{Some of this research was done as a part of  European Science Foundation Networking Programme HCAA.
The work of the first author was partially supported by the Israel
Science Foundation under the grant 688/08 The second author
gratefully acknowledges partial support from the National Science
Foundation grant DMS-0855597.}

\date{}


\maketitle

\begin{abstract}
In 1934 H. Malmheden \cite{M} discovered an elegant geometric
algorithm for solving the Dirichlet problem in a ball. Although his
result was rediscovered independently by Duffin \cite{D} 23 years
later, it still does not seem to be widely known. In this paper we
return to Malmheden's theorem, give an alternative proof of the
result that allows generalization to polyharmonic functions and,
also, discuss applications of  his theorem to geometric properties
of harmonic measures in balls in $\mathbb{R}^n$.
\end{abstract}

\section{Introduction}\label{sec1}

In 1934, H. Malmheden, a doctoral student of M. Riesz in Lund,
proved that the value of a harmonic function $u$ at any point $P$ in
a disk, or in a ball in $\mathbb{R}^n$, $n\ge 3$, can be computed
from its values on the boundary according to the following
algorithm. Take an arbitrary chord $L$ through $P$; calculate the
value at $P$ of the linear function $\ell$ on $L$ that interpolates
the values of $u$ at the endpoints of chord $L$; and, finally,
calculate the average of the values $\ell(P)$ over all chords $L$
through $P$ (cf. Theorem \ref{thm2.1} for the precise statement),
\cite{M}, also see \cite{B1,B2,B3,B4,B5}.

To fix the ideas, if we denote, say in $\mathbb{R}^2$, by
$A(\theta)$ the value at $P$ obtained by linearly interpolating the
boundary values of a harmonic function  $u$ in the disk at the
endpoints of the chord through $P$ making an angle $\theta$ with the
positive direction of the $x$-axis, Malmheden's theorem yields that
$u(P)=\left(1/2\pi\right) \int_{0}^{2\pi} A\left(\theta\right)\,
d\theta$. It follows then from Malmheden's theorem that
$u(P)\le\max\limits_\theta A(\theta)$. In \cite{B2,B3} J. Barta
attempted to extend the latter inequality to all convex, or even
star-shaped regions. In \cite{W}, Weinberger showed that no such
inequality, even allowing a multiplicative factor on the right, can
hold at all points of a convex region unless the region is a disk.
The analogous assertion is true in higher dimensions. Thus, the
converse to Malmheden's theorem also holds; i.e., the assertion of
his theorem is correct only if the domain is a ball. (Note, in
passing, that the complex-analytic version of Malmheden's theorem is
well-known under the name of Bochner - Martinelli formula. It
provides holomorphic extensions of analytic functions from the
boundaries of \textit{arbitrary} smoothly bounded domains in
$\mathbb{C}^n$ and is obtained by averaging  one-dimensional Cauchy
integral representations -- cf., e.g., \cite{GH}.)

With this paper we hope to return the beautiful result of Malmheden
to mathematical folklore since it appears that the result is not
widely known. Moreover, we present an alternative approach to
Malmheden's theorem that allows to generalize the result to the
polyharmonic functions, cf. \cite{B1,B4,B5}.We also present
applications of Malmheden's theorem yielding nice geometric
properties of the harmonic measure in balls in $\mathbb{R}^n$.

Regarding the history of the result, it must be mentioned that
Duffin \cite{D} has re-discovered Malmheden's theorem more than two
decades later. His proof, although obtained independently, is
essentially that of \cite{M}.

 The paper is organized as follows. In \S \ref{sec2} and \S
 \ref{sec3} we give two different  proofs of the Malmheden theorem and its
 generalization to $k-$dimensional cross-sections of the ball passing through a given point
 $P$. The first proof in \S \ref{sec2} is essentially Malmheden's original
 proof. The second proof seems to be new and is the crux for further generalizations of
Malmheden's procedure to polyharmonic functions.
 In \S \ref{sec4} we show the converse, i. e., that balls are the only domains in $\mathbb{R}^n$ for
 which Malmheden's theorem holds. In \S \ref{sec5} we discuss
 Malmheden's theorem in relation to some geometric properties of harmonic measure in
 the ball. In \S \ref{sec6}  we extend Malmheden's theorem
 to polyharmonic functions in $\mathbb{R}^n$. In $2$ dimensions this was done by Barta, cf. \cite{B1,B4,B5}; his proof is quite different from ours.
 A different converse theorem, assuming  that Malmheden's algorithm only reproduces harmonic functions at one point
 and, under additional condition of central symmetry, implying that the domain is a ball is given  in \S \ref{sec7}.
 We finish with some additional remarks in \S\ref{sec8}.

\section{A proof of Malmheden's theorem}\label{sec2}

Let us begin in a more general setting. Let $\Omega$ be a convex
open set in $\mathbb{R}^n$ with boundary $\Gamma$. Let $f$ be a
continuous function on $\Gamma$, $P\in\Omega$. Draw a chord $L$
through $P$ intersecting $\Gamma$ in exactly two points $Q_1$,
$Q_2$. Denote by $f_1$, $f_2$ the values of $f$ at $Q_1$, $Q_2$,
respectively. Interpolate the values $f_1$, $f_2$ by the unique
linear function on $L$ that we denote by $\ell$. Clearly, the value
$\ell(P)$ is a continuous function of $P$ and $L$. Holding $P$ fixed
let $u(P)$ denote the average of $\ell(P)$ over all lines $L$.
$u(P)$ is a continuous function in $\Omega$. Note, that if $f$ is a
restriction of a linear function onto $\Gamma$, then $f=u$ in
$\Omega$.

Clearly, $u(P)\to f\left(Q_0\right)$ when $P\to Q_0$,
$Q_0\in\Gamma$.

\begin{theorem}[Malmheden --- \cite{M,D}]\label{thm2.1}
If $\Omega$ is a ball in $\mathbb{R}^n$, then $u$ is harmonic in $\Omega$, and hence solves the Dirichlet problem (for the Laplacian) in $\Omega$
with data $f$.
\end{theorem}

\begin{proof}
In view of the preceding remarks it remains to show that $u$ is
harmonic. Let $L$, $P$, $Q_1$, $Q_2$ be as above,
$r_1=\left|PQ_1\right|$, $r_2=\left|PQ_2\right|$. Then
\begin{equation}
\label{eq2.1}
\ell(P)=\frac{r_1f_2+r_2f_1}{r_1+r_2}.
\end{equation}
First, assume $n=2$. Let $\theta$ be the (polar) angle that the
chord $L$ makes with the positive direction of the $x$-axis, so
$r_1(\theta +\pi)=r_2(\theta)$ and $f_1(\theta+\pi)=f_2(\theta)$.
Then, we have,
\begin{equation}
\label{eq2.2}
\begin{split}
u(P) &=\frac1{2\pi}\int_0^{2\pi}\ell(P)\,d\theta \\
&=\frac1\pi\int_0^{2\pi}\frac{r_2f_1}{r_1+r_2}\,d\theta.
\end{split}
\end{equation}

Let $\varphi$ be the angle between the (inner) normal $\vec{n}$ to
$\Gamma$ at $Q_1$ and the chord $L$. Then, comparing infinitesimal
arclengths on $\Gamma$ and the circle of radius $r_1$, centered at
$P$, we observe that $r_1d\theta=\cos\varphi\,ds$, where $ds$ is the
arclength on $\Gamma$. Thus, from \eqref{eq2.2} it follows that
\begin{equation}
\label{eq2.3}
u(P)=\frac1\pi\int_0^c\frac{\cos\varphi}{r_1}f_1\,ds
-\frac1\pi\int_0^c\frac{\cos\varphi}{r_1+r_2}f_1\,ds,
\end{equation}
where $c$ is the length of $\Gamma$.

The first term in \eqref{eq2.3} is harmonic since $\dfrac{\cos\varphi}{r_1}$ is nothing else but $\dfrac\partial{\partial n_{Q_1}}\log\left|P-Q_1\right|$,
 so the first integral is the ``double layer'' potential of $f_1$ (cf. \cite{K}). Finally, we use the hypothesis that $\Gamma$ is a circle of radius
 $R$. Then $r_1+r_2=2R\cos\varphi$, and the second integral is simply a
constant, hence $u$ is harmonic in $\Omega$.

In higher dimensions, the argument is almost identical. For the sake of clarity of notation, we take $n=3$. Then, \eqref{eq2.2} and \eqref{eq2.3} become
\begin{equation}
\label{eq2.2$'$}
u(P)=\frac1{2\pi}\int_\Gamma\frac{r_2f_1}{r_1+r_2}\,d\theta;
\tag{\ref{eq2.2}$'$}
\end{equation}
and, hence,
\begin{equation}
\label{eq2.3$'$}
u(P)=\frac1{2\pi}\int_\Gamma\frac{\cos\varphi}{r_1^2}\,f_1dS
-\frac1{2\pi}\int_\Gamma\frac{\cos\varphi}{r_1+r_2}\;\frac{f_1}{r_1}\,dS.
\tag{\ref{eq2.3}$'$}
\end{equation}
Here, $dS$ is the surface area measure on $\Gamma$ and $d\theta$ is
the solid angle measure at $P$. Once again, the first integral is
the double layer potential of $f_1$, while when $\Omega$ is a ball
of radius $R$, $r_1+r_2=2R\cos\varphi$, so the second integral
becomes a ``single-layer'' potential of $f_1$, harmonic in $\Omega$.

(Of course, \eqref{eq2.3$'$} and \eqref{eq2.3} correspond to very
well known representations of solutions of the Dirichlet problem in
the ball as sums of potentials of double and single layers in
dimensions $3$ and higher, and as (up to an additive constant) a
potential of a double layer in the disk (cf. \cite{K} or
\cite{KPS}).
\end{proof}

Let us sketch here a different proof of Theorem \ref{thm2.1} that
will allow us in \S\ref{sec6} to generalize Malmheden's algorithm to
biharmonic functions.

\begin{proof}[2nd Proof of Theorem \ref{thm2.1}]
It is more convenient to translate our coordinate system so the
origin is now at $P$ while
$\Omega=\left\{x\in\mathbb{R}^n:|x-\mathbf{c}|<1 \right\}$, where
$\mathbf{c}: \left|\mathbf{c}\right|<1$ is the center of the ball.
We wish to show that for any homogeneous harmonic polynomial $H$ in
$\partial\Omega$, Malmheden's procedure produces the number $H(0)$.
Clearly, this will imply the result since the restrictions of
harmonic polynomials to the sphere $\partial\Omega$ are dense in
$C(\partial\Omega)$, while linearity of Malmheden's algorithm yields
the result for all harmonic polynomials as long as it holds for
homogeneous polynomials. Finally, since Malmheden's theorem
obviously holds for linear polynomials we may assume that $m:=\deg
H>1$. and need to verify that the algorithm produces the number
$0=H(P)$. Fix $\mathbf{e}$ a unit vector in $\mathbb{R}^n$, and let
$L=L_{\mathbf{e}}$ be the line through the origin parallel to
$\mathbf{e}$. Let $t$ be the running coordinate on $L_{\mathbf{e}}$.
$L_{\mathbf{e}}$ meets $\partial\Omega$ at two points $Q_1$, $Q_2$
where $t$ is equal to the roots of the equation
\begin{equation}
\label{eq2.4} \left|t\mathbf{e}-\mathbf{c}\right|^2=1,\quad
t^2-2<\mathbf{c},
\mathbf{e}>t+\left(\left|\mathbf{c}\right|^2-1\right)=0.
\end{equation}
Along $L_{\mathbf{e}}$,
$H\left(t\mathbf{e}\right)=t^mH\left(\mathbf{e}\right)$. If we
denote by $a,b\in\mathbb{R}$ the values of $t$ corresponding to
$Q_1$, $Q_2$ and determined by \eqref{eq2.4}, then the linear
interpolant of $H$ along the segment of $L_{\mathbf{e}}$ between
$Q_1$, $Q_2$ equals
\begin{equation}
\label{eq2.5} \left\{\frac{t-b}{a-b}\,a^m
+\frac{t-a}{b-a}\,b^m\right\}H\left(\mathbf{e}\right),
\end{equation}
and is equal at $t=0$ to
\begin{equation}
\label{eq2.6}
(ab)\frac{\left[a^{m-1}-b^{m-1}\right]}{b-a}\,H\left(\mathbf{e}\right).
\end{equation}
Now, we note from \eqref{eq2.4} that $ab=|\mathbf{c}|^2-1$,
independent of the chosen vector $\mathbf{e}$. (Of course, this is a
well-known theorem from the euclidean geometry). So the cofactor of
$H\left(\mathbf{e}\right)$ in \eqref{eq2.6} is a symmetric
polynomial of degree $m-2$ of the roots $a,b$ of \eqref{eq2.4}.
Hence, by a standard result on symmetric functions, it is a
polynomial of degree at most $m-2$ of the coefficients of the
quadratic \eqref{eq2.4}. Now observe that the only coefficient
dependent on $\mathbf{e}$ in \eqref{eq2.4} is that of $t$ and it is
of degree $1$ in the coordinates $e_j$ of $\mathbf{e}$. Summarizing,
the cofactor of $H\left(\mathbf{e}\right)$ in \eqref{eq2.6} is a
polynomial in $e_j$ of degree at most $m-2$. Hence, the cofactor as
a function of $\mathbf{e}$ is orthogonal to the homogeneous harmonic
polynomial $H\left(\mathbf{e}\right)$ over the unit sphere, thus
Malmheden's procedure for $H$ indeed produces
$0=H\left(\mathbf{0}\right)$. The last assertion follows from the
fact that a polynomial of any degree $k$ (in our case, $k=m-2$)
matches on the sphere a harmonic polynomial of degree $\le k$,
which, of course, is the sum of homogeneous harmonic polynomials.
Homogeneous harmonic polynomials of different degrees are orthogonal
on the unit sphere, cf. e.g., \cite{F, K, KS}. The second proof is
now complete.
\end{proof}

\section{An extension of Malmheden's theorem to cross-sections of higher dimension}\label{sec3}

Theorem \ref{thm2.1} has an almost immediate generalization if one
replaces chords through a point $P$ by $k$-dimensional
cross-sections. (For the notational convenience we shall often
represent $P$ in the coordinate form as $x=(x_1,\dotsc, x_n)$). More
precisely, let $\Omega$ be a convex bounded domain in $\mathbb{R}^n$
with smooth boundary $\Gamma$. Let $P(x,y)$, $x\in\Omega$,
$y\in\Gamma$, denote the Poisson kernel for $\Omega$, so any
harmonic function in $\Omega$, continuous in $\overline{\Omega}$,
can be represented by its boundary values $f(y)$ on $\Gamma$ as
\begin{equation}
\label{eq3.1}
u(x)=\int f(y)\,P(x,y)\,dS(y).
\end{equation}
Let, as usual, the Green function in $\Omega$ be
\begin{equation}
\label{eq3.2}
g(x,y)=\begin{cases}
-\frac1{2\pi}\log|x-y|+u_x(y); & n=2;\\
\frac1{\omega_n}|x-y|^{2-n}+u_x(y); & n\ge 3,
\end{cases}
\end{equation}
where $u_x(y)$ is harmonic in $\Omega$, and
$g(x,y)\mid_{y\in\Gamma}\equiv 0$. $\omega_n$ denotes the area of
the unit sphere in $\mathbb{R}^n$. Then, as is well-known,
\begin{equation}
\label{eq3.3}
P(x,y)=\frac\partial{\partial n_y}\,g(x,y),
\end{equation}
when $n_y$ is the inner normal to $\Gamma$ at $y\in\Gamma$. For
$\Omega=\{|x|<1\}$,
\begin{equation}
\label{eq3.4} P(x,y)=\frac1{\omega_n}\;\frac{1-|x|^2}{|x-y|^n}
\end{equation}
(cf. \cite{K}, e.g., for more details).

For $n=1$ there is a formula playing the role of \eqref{eq3.1} and
\eqref{eq3.4}, where the ``domain'' is the interval $[-1 , 1]$. Then
\begin{equation*}
\frac12\left[f(1)\frac{1-x^2}{(1-x)}+f(-1)\frac{1-x^2}{(x+1)}\right]
=\frac12\left[f(1)(1+x)+f(-1)(1-x)\right]
\end{equation*}
is the linear function matching given data $f$ at $\pm 1$.

Fix $k$: $1\le k<n$ and consider a $k$-dimensional plane $\alpha$
through $P$. $\Omega_\alpha:=\Omega\cap\alpha$ is a convex domain in
$\mathbb{R}^k$. Let $P^\alpha$ denote the Poisson kernel for
$\Omega_\alpha$. Replacing in \eqref{eq3.1} $P(x,y)$ by
$P^{\alpha}$, $f$ by $f_\alpha=f\mid_{\alpha\cap\Gamma}$, and
$d\,S(y)$ by $dS_\alpha$, the Lebesgue measure on
$\alpha\cap\Gamma$, yields the solution of the Dirichlet problem
$P^{\alpha}(f_{\alpha})$ for $\Omega_\alpha$ with the data
$f_\alpha$ with respect to the $k$-dimensional Laplacian. Fix
$x\in\Omega$ and consider the Grassmann manifold $G(k,n)$ of all
$k$-dimensional planes $\alpha$ through $x$. The orthogonal group
$SO(n)$ acts naturally on $G(k,n)$ and the invariant Haar measure on
$SO(n)$ induces the unique normalized measure $dm_k(\alpha)$ on
$G(k,\alpha)$. $m_k(\alpha)$ is invariant with respect to all
rotations of $\mathbb{R}^n$ with the center $x$.

Now keeping the notations from \S \ref{sec2}, and in view of the
note following \eqref{eq3.4}, we can restate Theorem \ref{thm2.1}
and rewrite \eqref{eq3.1}. If $\Omega=\{|x|<1\}$ is the unit ball,
then any harmonic function $u$ in $\Omega$, continuous in
$\overline{\Omega}$ with $u\mid_\Gamma=f$ can be represented via
solutions of one-dimensional Dirichlet problems on $1$-dimensional
lines $L$ passing through $x\in\Omega$. Namely, denoting by
$f_L:=f\mid_{L\cap\Gamma}$, we have:
\begin{equation}
\label{eq3.5} u(x)=\int_{L\in G_{x}(1,n)}P^L \left(f_L\right)
dm_1(L).
\end{equation}
This yields an immediate corollary.

\begin{theorem}\label{thm3.1}
Let $\Omega=\{|x|<1\}$ be the unit ball, $f\in C(\Gamma)$,
$\Gamma=\partial\Omega$ is the unit sphere. The harmonic extension
$u$ of $f$ to $\Omega$ can be represented in terms of harmonic
extensions of $f$ into $k$-dimensional sections of $\Omega$:
\begin{equation}
\label{eq3.6}
u(x)=\int_{\alpha\in
G_x(k,n)}P^\alpha\left(f_\alpha\mid_{\Gamma\cap\alpha}\right)(x)\,dm_k(\alpha).
\end{equation}
\end{theorem}

\begin{proof}
From \eqref{eq3.5} and Fubini's theorem we get
\begin{equation*}
\begin{split}
u(x) &=\int_{L\in G_x(1,n)}P^L\left(f_L\mid_{\Gamma\cap L}\right) dm_1(L) \\
&=\int_{\alpha\in G_x(k,n)}\int_{L\in\alpha,x\in L}P^L\left(f_L\mid_{\Gamma\cap L}\right) dm_1(L)\,dm_k(\alpha) \\
&=\int_{\alpha\in G_x(k,n)}P^\alpha\left(f_\alpha\mid_{\Gamma\cap\alpha}\right) dm_k(\alpha).\quad\qed
\end{split}
\end{equation*}
\end{proof}

\section{The converse}\label{sec4}

\begin{theorem}\label{thm4.1}
If $\Omega$ is a convex domain and the Malmheden procedure described in \S \ref{sec2}, produces for each continuous function $f$
the solution of the Dirichlet problem with data $f$, then $\Omega$ must be a ball.
\end{theorem}

For the sake of clarity, we restrict ourselves to the cases $n=2,3$.
The extension to higher dimensions is straightforward.

We shall use the notation from \S \ref{sec2}. The hypothesis implies
that the function $u(P)$ given by \eqref{eq2.3} or \eqref{eq2.3$'$}
is a harmonic function of $P$. Note that the first term, as we've
observed in \S \ref{sec2}, is ALWAYS harmonic in $\Omega$ since it
is equal to the double layer potential. Thus the hypothesis yields
that the second term is harmonic as a function of point $P\in\Omega$
as well. Therefore the functions
\begin{equation}
\label{eq4.1}
\text{(i)}\ \ \frac{\cos\varphi}{r_1+r_2}\qquad
\text{(ii)}\ \ \frac{\cos\varphi}{r_1\left(r_1+r_2\right)}
\end{equation}
are harmonic (as functions of $P$) in $\Omega$.

Put the origin at the point $Q_1$. Let $P=r_1t$, where $t:|t|=1$ is
a point on the unit sphere centered at $Q_1$.

Let us represent $\Gamma:=\partial\Omega$ near $Q_2$ by its ``polar'' equation: $\rho=\rho(t)$, $\rho=\left|Q_1Q_2\right|$. Then, \eqref{eq4.1} yields that functions
\begin{equation}
\label{eq4.2} \text{(i)}\ \ \frac{\langle
t,n_{Q_1}\rangle}{\rho(t)}\qquad \text{(ii)}\ \
\frac1r_1\;\frac{\langle t,n_{Q_1}\rangle}{\rho(t)},
\end{equation}
where $\langle ,\rangle$ denotes the scalar product and $n_{Q_1}$ is
the (inner) normal to $\Gamma$ at $Q_1$, are harmonic functions of
$r_1$ and $t$.

In (i), the function depends on the polar angle $t$ only, and being harmonic, forces it to be linear. Yet, it must be single-valued, thus is a constant. Denote it by $\dfrac1{2R}$. Thus, $\rho(t)=2R\cos\left(\langle\left(t,n_{Q_1}\right)\rangle\right)$ which is an equation of the circle.

For (ii), a simple calculation, or a quick check with \cite[p.\ 141,
Ex.\  4]{K} yields that $\dfrac{\text{const}}{r_{1}}$
\,($\dfrac{\text{const}}{r_{1}^{n-2}}$, $n\ge 3$, in general) are
the only homogeneous harmonic functions of degree $-1$ ($2-n$,
respectively). Hence, again, from \eqref{eq4.2}, (ii) we infer that
$\dfrac{\left\langle t,n_{Q_1}\right)}{\rho(t)}=$ const, i.e.,
$\Gamma$ is a sphere.\qed

\begin{remark}\label{rem4.2}
Weinberger \cite{W} actually showed a stronger converse. To fix the ideas, we shall describe his result for $n=2$.
\end{remark}

We keep the same notation as above. If Malmheden's procedure
produces the solution of the Dirichlet problem, then putting the
origin at point $P\in\Omega$, assuming $\Omega$ to be convex, and
the boundary $\Gamma$ to be given by
$\Gamma:=\{r=r(\theta),0\le\theta\le 2\pi\}$, we can  rewrite
\eqref{eq2.2} from \eqref{eq2.1} as
\begin{equation}
\label{eq4.3} u(P)=\dfrac1{2\pi}\int_0^{2\pi}
\frac{r(\theta)u\left(r(\theta+\pi),\theta+\pi\right)+r(\theta+\pi)u(r(\theta),\theta)}{r(\theta)+r(\theta+\pi)}\,d\theta.
\end{equation}
Equation \eqref{eq4.3} implies
\begin{equation}
\label{eq4.4} u(P)\le\max\limits_\theta\frac{r(\theta)u(r(\theta
+\pi),\theta+\pi)+r(\theta+\pi)u(r(\theta),\theta)}{r(\theta)+r(\theta+\pi)}.
\end{equation}
Weinberger showed that if a weaker version of \eqref{eq4.4} with a
constant factor on the right hand side holds for any $P\in\Omega$,
$\Omega$ being a convex domain, then $\Omega$ is a disk. This result
refuted the conjecture of Barta \cite{B3,B4} that \eqref{eq4.4}
should hold in general convex domains.

\begin{remark}\label{rem4.3}
Going over the proof of Theorem \ref{thm4.1}, one immediately notes that the hypothesis can be somewhat weakened:
\begin{enumerate}
\item[(i)]  Instead of convexity, it suffices to assume that the domain $\Omega$ is star-shaped with respect to all points $Q$
in a neighborhood of the fixed boundary point
$Q_1\in\partial\Omega$. Equivalently, that placing the origin at any
such $Q$, the (smooth, say $C^2$) boundary $\partial\Omega$ is given
by the ``polar'' equation $\rho=\rho(t)$, where $t:|t|=1$ runs over
the upper hemisphere of the unit sphere.

\item[(ii)]  Instead of assuming that Malmheden's algorithm successfully solves the Dirichlet problem for any data $f$,
it suffices to assume that merely for any continuous data supported
in a neighborhood of $Q_1$. We omit the details.
\end{enumerate}
\end{remark}

\section{Discussing Malmheden's Theorem and Harmonic Measure}\label{sec5}

We keep the same notation as in \S \ref{sec2}. If $\Omega$ is a
convex domain in $R^n$ and $P$ is an interior point, the
\textit{metric ratio} $R_P$ associated with the pair $(\Omega ,P)$
is the function defined on $\partial\Omega$ by: $R_P (Q_1)  =
\frac{|PQ_2|}{|Q_1Q_2|},$ where $Q_2$ denote the second point of
intersection with $\partial\Omega$  of the line $L$ joining $P$ to
$Q_1$. As in \S\ref{sec2}, we note that $R_P$ is a  continuous
function on $\partial\Omega$ with values strictly between 0 and 1.
It is constant (i.e., identically equal to $1/2$) if, and only if,
$P$ is a center of symmetry of $\Omega$ (i.e., all chords through
$P$ are bisected there). Let $w_P$ denote the harmonic measure on
$\partial\Omega$ evaluated at $P$ and $A_P$ denotes the subtended
angle measure on $\partial\Omega$ with respect to $P$. Malmheden's
theorem then can be reformulated as follows.

\begin{theorem}\label{thm5.1}
 If $\Omega$ is a ball, then for any $P\in\Omega$ , $R_P$ is the
Radon - Nikodym derivative of the harmonic measure $w_P$ with
respect to $A_P$.
\end{theorem}

For the proof one only needs apply Theorem \ref{thm2.1} to the
characteristic functions of  arbitrary (say, relatively open)
subsets of $\partial\Omega$.

  Thus, Malmheden's theorem can be seen as a procedure to  (when $\Omega$ is a ball ) compute harmonic measure from ``purely geometric''
  quantities, lengths and (solid) angles, by integration. Note that the corresponding derivative
  when $A_P$ is replaced by surface measure on $\partial\Omega$ is given by
$C(P) / |PQ_1|^d$, the Poisson kernel \eqref{eq3.4} evaluated at $P$
and $Q_1$. For the unit ball centered at the origin $O$,
$C(P)=\frac{1}{\omega_n}(1-|OP|^2)$, where $\omega_n$ is the surface
area of the unit sphere -- cf. \cite{F, K} . Modulo the known
relation between measures $A_w$ and the surface area measure
Malmheden's theorem is thus equivalent to Poisson's formula for the
solution to Dirichlet's problem for the ball.

Against this background it is interesting to note the following
further connections between  harmonic and subtended angle measures
in the ball, mediated by Malmheden's theorem. It implies other
elegant properties of harmonic measure which, in dimensions $>2$,
are rarely noted.

\begin{corollary}\label{cor5.1}
Let $P$ be a point in the unit ball $\Omega :=\{x :
\left|x\right|<1\}$, and let a double cone $K_P$ with vertex at $P$
cut out ``spherical caps'' $U$ and $V$ from the unit sphere
$\mathcal{S}$. Let $w_P(E)$ denote the harmonic measure of a set
$E\subset S$ evaluated at $P$. Then, $w_P(U)+w_P(V)$ equals twice
the (normalized) solid angle at the vertex $P$ of $K$. (The
normalizing factor equals $\frac{1}{\omega_n}$, where $\omega_n$
denotes the surface area of the unit sphere in $\mathbb{R}^n$.)
\end{corollary}

Keeping the same notation as in Corollary \ref{cor5.1}, we also have
the following.

\begin{corollary}\label{cor5.2}
If we consider a system of masses consisting of caps $U$, $V$, each
endowed with the harmonic measure $w_P(U)$, $w_P(V)$, respectively,
then the center of mass is at $P$.
\end{corollary}

\begin{remark}\label{rem5.1}
Corollary \ref{cor5.1} is well known in $2$ dimensions --- cf.
\cite[Ch. IV, \S2]{N} and can be used to give an even shorter proof
of Malmheden's theorem. Corollary \ref{cor5.2}, even in the two
dimensional case, seems not to have been noticed before.
\end{remark}

\begin{proof}[Proof of Corollary \ref{cor5.1}]
Fix $P\in\Omega$. Let $J=J_P$ be the involution of $\mathcal{S}$
that maps a point $x$ of the unit sphere $\mathcal{S}$ onto the
second point where the line $Px$ meets $\mathcal{S}$. Let $f\in
C(\mathcal{S})$ be self-involutory with respect to $J$, i.e.
$f(x)=f(Jx)$, let us denote by $dA_P$ the normalized solid angle
measure at the vertex $P$. Then, Theorem \ref{thm2.1} implies that
\begin{equation}
\label{eq5.1} \int_{\mathcal S}f(x)\,dw_P(x) =\int_{\mathcal
S}f(x)\,dA_P(x).
\end{equation}

Thus, \eqref{eq5.1} holds for all self-involutory $f\in
C(\mathcal{S})$. But the characteristic function $\chi_{U\cup V}$ of
the union $U\cup V$ is obviously a bounded pointwise limit of the
self-involutory (w.r.t.\ $J$) functions in $C(\mathcal{S})$. Hence,
by the bounded convergence theorem \eqref{eq5.1} also holds for
$\chi_{U\cup V}$. This is the conclusion of Corollary \ref{cor5.1}.
\end{proof}

\begin{proof}[Proof of Corollary \ref{cor5.2}]
If $\ell(x)$ is a linear polynomial and $f\in
C(\mathcal{S}):f(x)=f(Jx)$ is self-involutory with respect to $J_P$,
then the Malmheden algorithm applied to $(\ell f)(x)$ (of course,
w.r.t.\  $P$) produces
\begin{equation}
\label{eq5.2} \ell(P)\int_{\mathcal{S}}f(x)\,dA_P(x).
\end{equation}
(This is seen at once since Malmheden's algorithm preserves linear
functions while $f(x)=f\left(J_Px\right)$).

By Theorem \ref{thm2.1} we have then
\begin{equation}
\label{eq5.3} \ell(P)\int_S
f(x)\,dA_p(x)=\ell(P)\int_{\mathcal{S}}f(x)\,dw_P(x).
\end{equation}
So, from \eqref{eq5.2}, we therefore infer that
$$
\int_{\mathcal{S}}\ell(x)\cdot f(x)\,dw_P(x)
=\ell(P)\int_{\mathcal{S}}f(x)\,dw_P(x).
$$
Hence, for all linear functions $\ell$, we have
\begin{equation}
\label{eq5.4} \ell(P)=\frac{\int\limits_S\left(\ell(x)\cdot
f(x)\right)\,dw_P(x)}{\int\limits_{\mathcal{S}}f(x)\,dw_P(x)}.
\end{equation}
Substituting for $\ell$ the coordinate functions $x_j, j=1,...,n$,
we obtain from \eqref{eq5.4} at once that the center of mass of the
mass density $f(x)\,dw_P(x)$ is at $P$. Taking for $f$ the
characteristic function $\chi_{U\cup V}$ of the union $U\cup V$
proves the corollary.
\end{proof}

\begin{remark}\label{rem5.2}

\begin{enumerate}
\item[(i)]  Corollary \ref{cor5.1} has an independent proof in $2$ dimensions. If $P$, with a slight abuse of notation, denotes a complex number
in $\mathbb{D}$, $\mathbb{D}=\{|z|<1\}$, it is straightforward to
calculate the involution $J_{P}(z)$ to be
$J_{P}(z)=\dfrac{P-z}{1-\bar{P}z}$.
 Consider an arc $U$ on the unit circle $\mathbb{T}$. Its harmonic measure at $P$ equals the harmonic measure of $V=J_{P}(U)$ evaluated at $0$
 (by conformal invariance of the harmonic measure, since $J_{P}(P)=0$). The harmonic measure of an arc $V$ evaluated at the origin equals the central
 angle associated with $V$ normalized by the factor $1/2\pi$. Thus, the harmonic measure $w_P(U)$ at $P$ equals
    \begin{equation}
    \label{eq5.5}
    w_P(U)=\frac{|V|}{2\pi},
    \end{equation}
    where $|V|$ denotes the length of the arc $V=J_{P}(U)$. Hence, recalling that $J_{P}(V)=U$, we conclude that
    \begin{equation}
    \label{eq5.6}
    \begin{gathered}
    w_P(U)+w_P(V)=\frac1{2\pi}\left(|U|+|V|\right) \\
    =\frac1\pi\left(A_P:=\text{ angle subtended by $U$ w.r.t. }P\right) \\
    =\frac1\pi\left(A_P:=\text{ angle subtended by $V$ w.r.t. }P\right).
    \end{gathered}
    \end{equation}
    (The last two equalities are, of course, corollaries of Euclid's theorem about   angles formed by the two chords through a point in the disk which states
     that their radian measure equals to the average of the radian measures of the subtending arcs of the circle.) This gives a different proof of Corollary \ref{cor5.1} in $2$
dimensions.

\item[(ii)]  From (i), one can easily deduce Theorem \ref{thm2.1} for $n=2$. Indeed, Corollary \ref{cor5.1} yields that for every $w\in\mathbb{D}$
with, as before, $A_w$ denoting the (normalized) subtended angle
measure as seen from $w$, and for every analytic polynomial $P(z)$,
we have
    \begin{equation}
    \label{eq5.7}
    \int_\mathbb{T} P(\xi)\,dA_w(\xi)
    =\frac12\left(P(0)+P(w)\right).
    \end{equation}
    (Recall that $J_w(z)=\dfrac{w-z}{1-\bar{w}z}$, $J_w(0)=w$, $J_w(w)=0$.) Obviously, \eqref{eq5.7} then holds for all functions $f$ analytic
    in $\mathbb{D}$ and continuous in $\overline{\mathbb{D}}$. Let $f$ be any such function. Parameterize the lines through $w$ by their intersection
    points $z$ with the unit circle $\mathbb{T}$. It is easy then by solving a system of linear equations to find for the line $\ell:=\ell(z)$
    the linear interpolant for the values $f(z)$, $f\left(J_w(z)\right)$. At $w$ it equals
    \begin{equation}
    \label{eq5.8}
    w\,\frac{f(z)-f\left(J_w(z)\right)}{z-J_w(z)}
    +\,\frac{z f\left(J_w(z)\right)-f(z)J_w(z)}{z-J_w(z)}.
    \end{equation}

    Clearly for a fixed $w\in\mathbb{D}$ and $f$ analytic, say, in a neighborhood of $\overline{\mathbb{D}}$ \eqref{eq5.8} produces an analytic function $F(z)$.
    Remembering that $J_w$ maps $0$ to $w$ and $w$ to $0$, one easily calculates that $F(0)=F(w)=f(w)$. Hence, applying Malmheden's procedure to $f(z)$,
    $z\in\mathbb{T}$ and then using $F(z)$ from \eqref{eq5.8} and \eqref{eq5.7} we obtain the value $f(w)$. Thus, Malmheden's algorithm applied to, say,
    any polynomial in $z$, produces the ``correct'' value at $w$.

Separating real and imaginary part yields Malmheden's theorem in the
disk, i.e., for $n=2$.
\end{enumerate}
\end{remark}

The following proposition can be viewed as a converse of Corollaries \ref{cor5.1} and \ref{cor5.2} combined.

\begin{proposition}\label{prop5.1}
Let $\Omega$ be a bounded convex domain in $\mathbb{R}^n$ and
$P\in\Omega$. Assume that the harmonic measure
$w_P=w(\cdot,\partial\Omega,P)$ on $\Gamma:=\partial\Omega$ at $P$
satisfies the conclusions of Corollaries \ref{cor5.1} and
\ref{cor5.2}. Then, $P$ is a ``Malmheden'' point for $\Omega$, i.e.,
Malmheden's algorithm (cf.\ \S\S 1 and 2) applied to any continuous
function $f\in C(\Gamma)$ on $\Gamma$ produces the value at $P$ of
the solution to the Dirichlet problem in $\Omega$ with data $f$ on
$\Gamma$.
\end{proposition}

\begin{proof}
Let $I(\Gamma)\subset C(\Gamma)$ consist of all self-involutory
functions, i.e., such that $f\in C(\Gamma):f(J(x))=f(x)$,
$x\in\Gamma$. (Here, $J=J_P$ is as above, the involution of $\Gamma$
induced by $P$.) Denote by $R\subset C(\Gamma)$ the subset of $f\in
C(\Gamma)$ for which the Malmheden algorithm produces the desired
value at $P$ of the solution of the Dirichlet problem with data $f$.
Our hypothesis implies that
$R\supset\{I(\Gamma)\}\cup\left\{(\ell\cdot f),f\in
I(\Gamma),\ell\text{ is a linear function of
}x_1,\dotsc,x_n\right\}$. We need to show that $R$ contains all
polynomials, i.e., is dense in $C(\Gamma)$. Since $R$ is obviously
closed the proposition will then follow.

Fix a linear function $\ell(x)$. Denote  $\ell^{\#}=\ell\circ J$. We
have
$\ell^2=\ell\cdot\left(\ell+\ell^{\#}\right)-\ell\cdot\ell^{\#}$
and, since $\ell+\ell^{\#}$, $\ell\cdot\ell^{\#}\in I(\Gamma)$, our
hypothesis implies that for any $g\in I(\Omega)$, $\ell^2\cdot
g=\ell\cdot\left(\ell+\ell^{\#}\right)g-\left(\ell\cdot\ell^{\#}\right)g\in\{I(\Gamma)\}\cup\{(\ell\cdot
f),f\in I\}\subset R$. So, $\left\{\ell^2\cdot
I(\Gamma)\right\}\subset R$. An induction argument shows that
$\left\{\ell^m\cdot I(\Gamma)\right\}\subset R$ for any integer $m$.
Since the linear span of the set of powers of linear polynomials
contains all polynomials the proposition follows.
\end{proof}
\begin{remark}\label{rem5.3}
It seems a worthy question whether the two geometric properties of
harmonic measure enunciated in Corollaries \ref{cor5.1} and
\ref{cor5.2} (i.e., the two hypotheses in Proposition \ref{prop5.1})
are actually independent of each other. We suspect they are but
haven't been able to prove it.
\end{remark}

\section{Malmheden's theorem for polyharmonic functions}\label{sec6}

Let $\Omega$ be the unit ball in $\mathbb{R}^n$.
$\Gamma:=\partial\Omega$ is the unit sphere, and let $f:=f(x)$ be a
smooth function, say even real-analytic, in a neighborhood of
$\Gamma$. Then as is well-known (cf. \cite{F}) the Dirichlet problem
for the bi-harmonic operator
\begin{equation}
\label{eq6.1}
\begin{cases}
\Delta^2u=0; \\
u=f, & \nabla u=\nabla f\text{ on }\Gamma
\end{cases}
\end{equation}
has a unique solution in $\Omega$. Not going into technicalities,
the reader may argue as follows.

As is well-known --- cf. \cite{ACL}, any function $u:\Delta^2u=0$ in
a domain $\Omega$ admits so-called Almansi expansion:
\begin{equation}
\label{eq6.2} u=h_1+|x|^2h_2,
\end{equation}
with $h_1$ and $h_2$ harmonic functions in $\Omega$, uniquely
defined by $u$. One can trivially adjust \eqref{eq6.2} when
$\Omega=\{|x|<1\}$ is a ball and rewrite \eqref{eq6.2} as
\begin{equation}
\label{eq6.3} u=H_1+\left(|x|^2-1\right)H_2,
\end{equation}
$H_1,H_2$ being harmonic in $\Omega$. Then, to solve the BVP
\eqref{eq6.1}, we need to solve two consecutive Dirichlet problems:
first, for $H_1$ ($=f$ on $\Gamma$), and then for $H_2$:
$\frac\partial{\partial r}(f-H_1)=2H_2$, where
$\frac\partial{\partial r}$ stands for  the radial derivative.
Uniqueness follows by application of Green's formula making use of
the bi-harmonic Green's function --- cf., e.g., \cite{HK}.

Now to formulate the analogue of Malmheden's theorem for biharmonic
functions (i. e., those  for which $\Delta^2=0$) we proceed as
follows. The one-dimensional analog of the bi-Laplacian is the
operator $\left(\frac d{dx}\right)^4$, whose kernel consists of
cubic polynomials. We can extend Malmheden's procedure for a fixed
point $P\in \Omega$ to solutions of \eqref{eq6.1}. Again draw a
chord $L$ through $P$ intersecting $\Gamma=\partial \Omega$ at
points $Q_1,Q_2$. Let $t$ denote the real parameter along $L$, $a$
and $b$ being the values of $t$ at $Q_1,Q_2$. Let $C(t)$ denote the
(unique) cubic polynomial such that the functionals $C(a)$, $C'(a)$,
$C(b)$, $C'(b)$ interpolate the corresponding values of the data
$f_L(t):=f\mid_L$. (For derivatives of $f_L$ at $a$ and $b$ we take
the values of the directional derivatives of the data $f$ along $L$
at those points.) $C_L(P)$ is then the value of $C$ at $P$. Holding
$P\in \Omega$ fixed, let $u(P)$ denote the average of $C_L(P)$ over
all lines $L$. $u(P)$ is, of course, a continuous function in $B$.
Also, if $f(x)$ is a cubic polynomial, then $u=f$ in $\Omega$.
Clearly, $u(P)\to f\left(Q_o\right)$, $\nabla u(P)\to\nabla
f\left(Q_o\right)$ when $P\to Q_o$, $Q_o\in\Gamma$.

\begin{theorem}[\cite{B1,B5}]\label{thm6.1}
If $\Omega$ is a ball in $\mathbb{R}^n$, $u$ is biharmonic in
$\Omega$, and hence solves \eqref{eq6.1}.
\end{theorem}

\begin{proof}
It is convenient to translate the coordinate system so $\Omega$ is a
unit ball centered at some point $\mathbf{c}$, while $P$ is now  the
origin. By a standard approximation argument we may assume that the
data $f(x)$ is a polynomial since the latter are dense in the space
$C^1(\Gamma)$ of continuously differentiable functions on $\Gamma$.
Moreover, we may also assume that $f(x)$ is a biharmonic polynomial
since the latter are also dense in the space of smooth biharmonic
functions in the  $C^1$-metric. Lastly, by linearity of the
Malmheden operator, we may assume that $f(x)$ is in fact a
homogeneous biharmonic polynomial of degree $m>3$.

Let us separate the following one-dimensional interpolation lemma.

\begin{lemma}\label{lem6.2}
Let $a<0<b$, $m>3$ and $C(t)=C_m(t)$ be the (unique) cubic
polynomial interpolating the values of $t^m$ and its first
derivative at $a$ and $b$. Then, $C_m(0)$ is a homogeneous symmetric
polynomial of degree $m$. Moreover,
\begin{equation}
\label{eq6.4} C_m(0)=(ab)^2q_{m-4}(a,b),
\end{equation}
where $q_{m-4}$ is a homogeneous polynomial of $a$ and $b$ of degree
$m-4$.
\end{lemma}

Assuming the lemma and in view of our chain of reductions we need
only show that for $f(x)=:H(x)$, $H(x)$ being a homogeneous
biharmonic polynomial of degree $m\ge 4$, the Malmheden algorithm
produces the number $0=H(P)$, (Point $P$, recall, is at the origin).
Let $L$, $a$, $b$ be as described in the paragraph preceding Theorem
\ref{thm6.1}. Note that $ab=$ constant that depends on $P$ only as
already noted in \S \ref{sec2} (cf. the argument below
\eqref{eq2.6}).
From this, \eqref{eq6.4} and since $H$ is a homogeneous polynomial
it follows, by repeating the argument following \eqref{eq2.6}, that
for each line $L$ through $P=\mathbf{0}$ defined by the directional
vector $\mathbf{e}\in\Gamma$, Malmheden's procedure produces
\begin{equation}
\label{eq6.5} \operatorname{const}q_{m-4}\left(\mathbf{e}\right)
H\left(\mathbf{e}\right)
\end{equation}
where the constant only depends on point $P$.

Averaging \eqref{eq6.5} over all chords $C$, i.e. over all
directional unit vectors $\mathbf{e}\in\Gamma$ produces $0$ as
required. To verify this last assertion we simply note, since $H$ is
homogeneous, \eqref{eq6.2} implies that
\begin{equation}
\label{eq6.6} H(x)=h_1(x)+|x|^2h_2(x),
\end{equation}
where $h_1,h_2$ are homogeneous harmonic polynomials of degrees $m$
and $m-2$ respectively. $q_{m-4}$, (as is true for any polynomial
--- cf. e.g., \cite{KS}), can be matched on the unit sphere $\Gamma$
by a harmonic polynomial $h_q$, $\deg h_q\le m-4<m-2$. From the
well-known orthogonality of spherical harmonics (see, e.g.,
\cite{KS}, \cite{ACL}, \cite{F},\cite{K}), we obtain our last
assertion and, hence, Theorem \ref{thm6.1} follows modulo Lemma
\ref{lem6.2}.
\end{proof}

\begin{proof}[Proof of Lemma \ref{lem6.2}]
Let $C_m(t):=C(t)=At^3+Bt^2+Ct+D$ match the function $t^m$ together
with its first derivatives at $a$ and $b$, $a<0<b$.

Writing down four equations corresponding to these $4$ interpolating
conditions we see at once (by Cramer's Rule) that all the
coefficients are \emph{rational} functions of $a,b$ of degree at
most $m+6$ symmetric with respect to $a$ and $b$ (interchanging $a$
and $b$ merely permutes equations in the system).

Furthermore, solving this linear system of equations via Cramer's
Rule we also observe that each coefficient $A,\dotsc,D$ is a
rational function of $a$ and $b$ for which the degree of the
numerator is at most $m+6$, while the degree of the denominator
which is the determinant of the system is precisely $6$. Now these
rational functions actually cannot have any finite poles for some
complex values of $a$ and $b$ since Hermite interpolation polynomial
$C_m(t)$ exists and is unique for all complex values $a,b$. Thus,
all the zeros of denominators in the rational expressions for the
coefficients $A,\dotsc,D$ of $C_m(t)$ must cancel out. Hence, all
the coefficients $A,\dotsc,D$ of $C_m(t)$ are actually polynomials
of degree $(m+6)-6=m$. Finally, since when $a$ or $b=0$, $D=D(a,b)$
must have a double zero at the origin, it follows that
$D(a,b)=(ab)^2q_{m-4}(a,b)$, where $q_{m-4}$ is a symmetric
polynomial in $a,b$ of degree $\le m-4$. The lemma is proved and the
proof of Theorem \ref{thm6.1} is now complete.
\end{proof}
\vspace{.5in}

\begin{remark}\hfill
\begin{enumerate}
\item[(i)]  Obviously, the idea of the above argument originated from our second proof of Malmheden's original
result --- Theorem \ref{thm2.1} --- given in \S\ref{sec2}.

\item[(ii)]  Lemma \ref{lem6.2} has a natural extension to higher order differential operators $\left(\dfrac d{dt}\right)^{2k}$, $k\ge 3$.
Accordingly, with rather obvious modifications Theorem \ref{thm6.1}
extends to polyharmonic operators $\Delta^k$ --- cf. \cite{B1,B5}.

\item[(iii)]  Theorem \ref{thm3.1} also readily extends to polyharmonic operators.

\item[(iv)]  We do not know whether the converse to Theorem \ref{thm6.1}, similar to Theorem \ref{thm4.1}, also holds for polyharmonic functions.
\end{enumerate}
\end{remark}


\section{Another converse to Malmheden's theorem}\label{sec7}

Let, as before, $\Omega$ be a convex, bounded domain in
$\mathbb{R}^n$, $\Gamma=\partial\Omega$, $P\in\Omega$. If
Malmheden's procedure with respect to $P$ as described in \S2
applied to any, say, harmonic polynomial $h$ yields the value
$h(P)$, it does not seem to imply that $\Omega$ is a ball. The
problem, of course, is that it does not allow us to locate the
center of the ball. However if one assumes that not only Malmheden's
procedure applied to harmonic functions $u$ in $\Omega$  yields the
correct value $u(P)$, but also $\Omega$ is centrally symmetric with
respect to $P$, one can conclude that $\Omega$ is a ball centered at
$P$.

The driving force for this result is the following extremely simple
but useful observation regarding harmonic measures.

\begin{lemma}\label{lem7.1}
Let $\Omega_1\not\equiv\Omega_2$ be two smoothly bounded,
star-shaped domains in $\mathbb{R}^n$ and assume that
$\Omega_1\cap\Omega_2\ne\emptyset$ and
$\partial\Omega_1\cap\partial\Omega_2\ne\emptyset$. Let
$O\in\Omega_1\cap\Omega_2$ be a point in $\Omega_1\cap\Omega_2$. Let
$w_1,w_2$ be harmonic measures on $\partial\Omega_1$,
$\partial\Omega_2$, respectively, evaluated at $O$. Let
$E_1=\partial\Omega_1\setminus\Omega_2$ be the portion of
$\partial\Omega_1$ that lies outside $\Omega_2$, while
$F_2=\partial\Omega_2\cap\Omega_1$ is the portion of
$\partial\Omega_2$ that lies inside $\Omega_1$. Similarly, define
$E_2=\partial\Omega_2\setminus\Omega_1$, and
$F_1=\partial\Omega_1\cap\Omega_2$.

Then,
\begin{equation}
\label{eq7.1} w_1\left(E_1\right)<w_2\left(F_2\right),
\end{equation}
and, similarly,
\begin{equation}
\label{eq7.2} w_2\left(E_2\right)<w_1\left(F_1\right).
\end{equation}

(Both inequalities are strict.)
\end{lemma}

\begin{proof}
We just prove \eqref{eq7.1}; the proof of \eqref{eq7.2} is
identical.

Let $U=\Omega_1\cap\Omega_2$, $\partial U=F_1\cup F_2$. On $\partial
U\setminus\left(F_1\cap F_2\right)$ the (harmonic) functions
$w_1(E_1;x)$, $w_2\left(F_2;x\right)$ satisfy (by the maximum
principle) the following:
\begin{equation}
\label{eq7.3}
\begin{gathered}
w_1(E_1;x)\mid_{F_1}=0;\quad w_1\left(E_1;x\right)\mid_{F_2}<1 \\
w_2\left(F_2;x\right)\mid_{F_1}>0;\quad
w_2\left(F_2;x\right)\mid_{F_2}=1.
\end{gathered}
\end{equation}
Thus,
\begin{equation}
\label{eq7.4} w_2\left(F_2;x\right)>w_1\left(E_1;x\right)
\end{equation}
on $\partial U\setminus\left(F_1\cap F_2\right)$. Since $F_1\cap
F_2$ has measure zero on $\partial U$, and hence its harmonic
measure is zero as well, the generalized maximum principle --- cf.,
e.g., \cite{F, K, N} --- for bounded harmonic functions yields that
\eqref{eq7.4} holds everywhere in $U$. This proves \eqref{eq7.1}.
\end{proof}

\begin{theorem}\label{thm7.2}
Let $\Omega$ be star-shaped domain in $\mathbb{R}^n$ and
$O\in\Omega$ be a point in $\Omega$. Assume also  that $\Omega$ is
centrally symmetric w.r.t. $O$. If the Malmheden algorithm applied
to $O$ reproduces the values at $O$ of all functions harmonic in
$\Omega$ and continuous in $\overline{\Omega}$, then $\Omega$ is a
ball centered at $O$.
\end{theorem}

\begin{proof}
First, observe that Lebesgue dominated convergence theorem
immediately extends the hypothesis to all bounded harmonic functions
whose boundary values are pointwise limits of continuous functions
on $\partial\Omega$. Thus, in particular, the conclusion applies to
harmonic measures of smoothly bounded open subsets of
$\partial\Omega$ (or, to subarcs in $2$ dimensions). Let $B$ be the
ball centered at $O$ with same volume as $\Omega$. If $B=\Omega$
there is nothing to prove. Then, $B\ne\Omega$ and since
$\operatorname{Vol}(B)=\operatorname{Vol}(\Omega)$, $\partial
B\cap\partial\Omega\ne\emptyset$. Since the Malmheden algorithm
applies to $\Omega$ at $O$, and $\Omega$ is centrally symmetric with
respect to $O$, the proof of Corollary \ref{cor5.1} implies that the
harmonic measure on $\partial\Omega$ evaluated at $O$ is identical
with the normalized solid angle measure subtended from $O$. Thus,
applying Lemma \ref{lem7.1} to the configuration $\Omega_1=\Omega$,
$\Omega_2=B$ and $O\in\Omega\cap B$, we arrive at the contradiction,
since both harmonic measures on $\partial\Omega$ and $\partial B$ at
$O$ equal to the solid angle measure subtended from $O$ and hence,
must be the same for the sets $E_1=\partial\Omega\setminus B$ and
$F_2=\partial B\cap\Omega$ and, respectively, for $E_2=\partial
B\setminus\Omega$ and $F_1=\partial\Omega\cap B$. This contradicts
\eqref{eq7.1}--\eqref{eq7.2}.

Therefore, $\Omega$ must equal $B$ and the theorem is proved.
\end{proof}

\section{Concluding remarks}\label{sec8}

\begin{enumerate}
\item[(i)]  Theorems \ref{thm2.1} and \ref{thm3.1} admit a nice
probabilistic interpretation, e. g., in $\mathbb{R}^3$. Informally,
it reduces to the following. Consider three ``Brownian travelers''
departing from a point $P$ in the unit ball $B$ in $\mathbb{R}^3$.
The first moves according to the laws of standard Brownian motion,
the second chooses at random a plane through $P$ and follows the
Brownian motion in that plane; the third chooses at random a line
through $P$ and follows the Brownian motion on that line. All of
these travelers will reach the unit sphere $\mathcal{S}:=\partial B$
with probability $1$. Malmheden's theorem asserts that for any
portion $E\subset \mathcal{S}$, the probability that the first
contact with $\mathcal{S}$ occurs in the set $E$ is the same for all
three travelers, i.e., the observer registering the exiting
travelers has no way of knowing how they arrived to the unit sphere
from $P$.

\item[(ii)] Theorem \ref{thm2.1} and, a more general Theorem \ref{thm3.1},  certainly suggest
connections to integral geometry. Indeed, the Malmheden algorithm
reminds of the inversion formula for Radon transform which
reconstructs functions from their integrals over hyperplanes, or
more generally, $k$-planes. According to this formula, the value of
a function at a point coincides, after applying  a certain power (a
half integer in even dimensions) of the Laplace operator to the
average of the Radon data through the point.

The Radon inversion formula delivers representation of functions as
continuous sums (i.e., integrals) of so-called plane waves, i.e., of
functions which are constant on families of parallel planes.
Similarly, Theorem \ref{thm3.1} can be interpreted in an analogous
manner with the role of plane waves played by "harmonic $k$-waves".
Here, by harmonic $k$-waves we understand functions which are
harmonic on parallel $k$- dimensional planes.

More precisely, fix the natural number $k, \ 1 \leq k \leq n.$
Denote by $\Delta_k$ the partial Laplace operator acting only on the
first $k$ variables:
$$\Delta_k=\sum\limits_{j=1}^{k}\frac{\partial^2}{\partial  x_j^2}.$$
For every rotation $\omega \in SO(n)$ we denote by
$\Delta_k^{\omega}$ the "`rotated"' partial Laplacian:
$$(\Delta_k^{\omega}g)(x):=\Delta_k(g\circ \omega^{-1})(\omega x).$$
It is not hard to show then that the Laplace operator coincides with
the average of the rotated partial Laplacians, i.e.,:
$$\Delta g(x)=\frac{n}{k}\int\limits_{SO(n)}\Delta_k^{\omega}g(x)d\omega,$$
where $d\omega$ denotes the normalized Haar measure on the
orthogonal group. Malmheden theorem establishes a similar link
between  the solutions of the corresponding Laplace equations:

\begin{theorem}\label{thm8.1}
Let $\Omega :=\{x : \left|x\right|<1\}$ be the unit ball and $f \in
C(\partial \Omega)$ and, for every $\omega \in SO(n)$, let
$u_{\omega}$ denote the (unique) solution of the boundary value
problem:
\begin{equation*}
\begin{cases}
\Delta_k^{\omega}u_{\omega}(x)=0, \, \, x \in \Omega,\\
u_{\omega}(x)=f(x), \, \, x \in \partial\Omega.
\end{cases}
\end{equation*}

Then the function
\begin{equation}\label{eq8.1}
u(x)=\int\limits_{SO(n)}u_{\omega}(x)d\omega
\end{equation}
solves the boundary value problem
\begin{equation*}
\begin{cases}
\Delta u(x)=0, \, \, x \in \Omega,\\
u(x)=f(x), \ \ x \in \partial\Omega.
\end{cases}
\end{equation*}
\end{theorem}

\begin{proof}
By construction,  the function $u_{\omega}\circ \omega^{-1},  \
\omega \in SO(n)$, is harmonic with respect to the first variables
$x_1, \cdots,x_k$, i.e., it is harmonic on every cross-section of
the ball by a $k$-dimensional plane parallel to the $k$-dimensional
plane $ \Pi =\{x_{k+1}=\cdots= x_n=0.\}$.
Then, $u_{\omega}$ is harmonic on each $k$-plane parallel to $\omega
\Pi.$ Let $x \in \Omega$ and let $\Pi_{\omega,x}$ be the
$k$-dimensional plane parallel to $\omega \Pi$ and passing through
$x$. Since $u_{\omega}=f$ on the unit sphere, $u_{\omega}$ is the
harmonic extension of $f$ into $\Pi_{\omega,x} \cap \Omega.$ When
$\omega$ runs over the whole orthogonal group $SO(n),\,
\Pi_{\omega,x}$ runs over all $k$-dimensional planes passing through
$x$ and Theorem 3.1 simply claims that the average, with respect to
all rotations $\omega$, of the values $u_{\omega}(x)$ equals to the
value at $x$ of the harmonic extension of the function $f$ into the
ball $\Omega$, i.e. $u(x).$ This proves (\ref{eq8.1}).
\end{proof}
The extended Malmheden's theorem for polyharmonic functions --cf.,
e.g., Theorem \ref{thm6.1} for biharmonic functions with $k=1$ --
allows a similar interpretation. In this case, the boundary
conditions will be corresponding complete sets of Cauchy data for
the polyharmonic equation $\Delta^N u=0$, i.e., the prescribed
values of functions and their first $N-1$ normal derivatives.

It would be interesting to investigate further whether a similar
decomposition perhaps holds for other differential operators with
constant coefficients and rotational symmetry, i.e.,  some operators
of the form $P(\Delta)$ where $P$ is a polynomial.

\item[(iii)]  There are various levels at which the converse to the Malmheden theorem can be formulated. Theorem \ref{thm4.1}
(arguably, the most natural one) assumes that Malmheden's procedure
produces the solution to the Dirichlet problem with arbitrary data
at all points of a convex domain $\Omega$. The conclusion is then
that $\Omega$ is a ball. Theorem \ref{thm7.2} is an attempt to
obtain the same conclusion under weaker assumptions: $\Omega$ is
assumed to be star-shaped and Malmheden's procedure is only assumed
to produce the desired results at one point. However, it required an
extra assumption of central symmetry. There are several other venues
of interest one may pursue here. To fix the ideas, let $\Omega$ be a
convex domain in $\mathbb{R}^n$, $P\in\Omega$, a fixed point.
$J=J_P$, as before, is the involution of the boundary
$\Gamma:=\partial\Omega$ determined by chords through $P$. As in
\S\ref{sec5}, $I=I_P(\Gamma)$ denotes the subspace of functions in
$C(\Gamma)$ invariant under $J$. We have noted earlier in
\S\ref{sec5}, that if $P$ is a Malmheden point of $\Omega$ (i.e.,
Malmheden's algorithm applied to every $f\in C(\Gamma)$ produces the
value $D_{P}f$ at $P$ of the solution to the Dirichlet problem for
the Laplacian with data $f$), then the following hold.
    \begin{enumerate}
    \item[(a)]  The two measures on $\Gamma$, $dA_{P}=$ the normalized subtended angle from $P$ and $w_P:=w(\cdot,\Gamma,P):=$ the harmonic measure
    on $\Gamma$ at $P$ produce the same results acting on all functions in $I_{P}(\Gamma)$. (Incidentally, this is equivalent to the relation between
    the solid angles of double cones with vertex at $P$ and the harmonic measure of the two surface portions they cut out of
     $\Gamma$ --- cf. \S\ref{sec5}, Corollary \ref{cor5.1}).

    \item[(b)]  For every function $f\in I_P(\Gamma)$ and every linear polynomial $\ell$, it follows that $D_P(\ell\cdot f)=\ell(P)\cdot D_P(f)$ ---
    cf. Corollary \ref{cor5.2}, Proposition \ref{prop5.1}.
    \end{enumerate}
    Proposition \ref{prop5.1} proves the converse, namely: if $\Omega,P$ are such that (a) and (b) hold, then $P$ is an $M$-point of $\Omega$.
    However, we do not know the exact relationship between (a) and (b).

\item[(iv)]  Property (a) of a domain $\Omega$ is on its own somewhat a mystery and holds, perhaps, the key to a deeper understanding of Malmheden's theorem.
Theorem \ref{thm7.2} yields that (a), together with the additional
hypothesis that $\Omega$ is  symmetric about point $P$, imply that
$\Omega$ is a ball centered at $P$. Without this assumption the
conclusion fails.
A similar problem in two dimensions with the subtended angle measure
replaced by a weighted arclength measure was settled in \cite[Thm.
3.29(i)]{EKS}, where it was shown that for those problems there are
other solutions besides circles.  We must exercise similar  caution
for our problem as well, since there are star-shaped domains
$\Omega$ (e.g., bicircular curves) in the plane for which the
subtended angular measure from a point $A\in\Omega$ equals to the
harmonic measure $w_B$ evaluated at another point $B\in\Omega$.
\end{enumerate}

The following example is due to the third author (H. S. Shapiro, 2006, unpublished).

\begin{proposition}\label{prop8.1}
There exists a planar domain $\Omega$,  star-shaped with respect to
the origin, and a point $P\ne 0$, $P\in\Omega$, so that for any
subarc $U\subset\Gamma$, $\Gamma:=\partial\Omega$, the angle that
$U$ subtends at the origin is equal to $2\pi w_P(U,\Omega,P)=2\pi
w_P(U)$. (As in previous sections, $w_P(U)$ denotes the harmonic
measure on $\Gamma$ evaluated at $P$.)
\end{proposition}

\begin{proof}
Consider the polynomial $q(z):=az^2+z+a$, $0<a<1/2$. $q$ is
univalent in the unit disk $\mathbb{D}$ and maps it conformally onto
a simply connected domain $\Omega$. For $z\in\mathbb{T}$,
$z=e^{i\theta}$, we have
$q(z)=z\left[a\left(z+\frac1z\right)+1\right]=\left(1+2a\cos\theta\right)z$.
Since $1+2a\cos\theta>0$ for all $\theta$, $q$ preserves arguments
of each $z\in\mathbb{T}$. Thus $\Omega$ is star-shaped and,
moreover, for any subarc $U\subset\Gamma$, the angle subtended by
$U$ at the origin $O$ is equal to the angle that its preimage
$U'=q^{-1}(U)$, $U\subset\mathbb{T}$ subtends at $O$. The latter, of
course, equals to the length $\left(U'\right)=2\pi w\left(U',
\mathbb{D},O\right)=2\pi w_O\left(U'\right)$, the harmonic measure
of $U'$ in $\mathbb{D}$ evaluated at the origin. By the conformal
invariance of harmonic measure, $2\pi w(U',\mathbb{D},O)=2\pi
w(U,\Omega,q(0))$ and setting $P=q(0)$ ($=a$) we are done.
\end{proof}

B. Gustafsson (a personal communication) has recently characterized
all simply connected domains $\Omega\subset\mathbb{C}$, where the
subtended angle measure from an interior point $A\in\Omega$ equals a
linear combination of the harmonic measures at finitely many other
(fixed) interior points. But this is the beginning of another tale.

\bibliographystyle{amsplain}
\bibliography{malmheden_thm}

\end{document}